\documentclass[11pt,reqno]{amsart}

\usepackage[margin=1in]{geometry}
\usepackage{amsmath,amssymb,amsthm,mathtools}
\usepackage{microtype}
\usepackage[hidelinks]{hyperref}

\newtheorem{theorem}{Theorem}[section]
\newtheorem{proposition}[theorem]{Proposition}
\newtheorem{lemma}[theorem]{Lemma}
\newtheorem{corollary}[theorem]{Corollary}
\theoremstyle{remark}
\newtheorem{remark}[theorem]{Remark}

\DeclareMathOperator{\Rm}{Rm}
\DeclareMathOperator{\Ric}{Ric}
\DeclareMathOperator{\inj}{inj}
\DeclareMathOperator{\diam}{diam}
\newcommand{\dGH}{d_{\mathrm{GH}}}

\title[Second-Order Departure of the Gigli--Mantegazza Flow]
{Second-Order Departure of the Gigli--Mantegazza Flow from Ricci Flow}
\author{Dongwoo Gang}
\address{Department of Mathematical Sciences, Seoul National University}
\email{dongwoo.gang@snu.ac.kr}
\date{August 14, 2026}
\makeatletter
\@namedef{subjclassname@2020}{\textup{2020} Mathematics Subject Classification}
\makeatother
\subjclass[2020]{53E20, 58J65, 49Q22}
\keywords{Gigli--Mantegazza flow, Ricci flow, heat kernel, optimal transport,
Wasserstein geometry}

\begin{document}
\raggedbottom
\begin{abstract}
For a closed connected Riemannian manifold $(M,g)$, the
Gigli--Mantegazza construction pulls back the quadratic Wasserstein metric
under the heat kernel embedding
$x\mapsto p_t(x,\cdot)\,d\operatorname{vol}_g$.  The resulting family
$\widetilde g_t$ agrees with Ricci flow to first order in $t$, but in
general not to second order.  We prove that
\[
 \widetilde g_t
 =g-2t\Ric_g
 +t^2\left(-\Delta\Ric_g+2\Ric_g^2-\frac23\mathcal Q_g\right)
 +O_{C^0}(t^3),
\]
where $\mathcal Q_g$ is quadratic in the full curvature tensor.  The term
$-\Delta\Ric_g$ also occurs in the second-order expansion of Ricci flow,
so the discrepancy depends pointwise and quadratically on the curvature.
In particular, at a Ricci-flat metric that is not flat, Ricci flow is
stationary while $\widetilde g_t$ is not.  The Gromov--Hausdorff distance
between the Gigli--Mantegazza and Ricci-flow metrics is $O(t^2)$, and round
spheres show that this estimate is sharp.
\end{abstract}
\maketitle

\section{Introduction and main results}
\label{sec:intro}

The relation between Ricci curvature, heat flow, and optimal transport is
well established.  Throughout this paper, we let $(M^n,g)$ be a closed
connected Riemannian manifold.  Von Renesse and Sturm
\cite{vonrenesse2005transport} proved that $\Ric_g\ge Kg$ holds if and
only if the heat flow on the quadratic Wasserstein space
$(\mathcal P_2(M),W_2)$ is $e^{-Kt}$-contractive.  McCann and Topping
\cite{mccann2010ricci} characterized backward super Ricci
flows by the contractivity of the corresponding diffusion in the
time-dependent Wasserstein distance.  Along a
backward Ricci flow, the evolution of the metric compensates for any lack
or excess of contractivity of the diffusion.

Motivated by this connection, Gigli and Mantegazza \cite{gigli2014flow}
introduced a different construction using the heat semigroup of the
initial metric $g$.  Let $P_t=e^{t\Delta_g}$ be that semigroup and let
$p_t(x,y)$ be its kernel with respect to $d\operatorname{vol}_g$.  For
$t>0$, the map
\[
 x\longmapsto\mu_{t,x}:=p_t(x,\cdot)\,d\operatorname{vol}_g
\]
embeds $M$ into $(\mathcal P_2(M),W_2)$ and induces the pullback metric
$\widetilde g_t$ on $M$.
Thus $\widetilde g_t$ records the infinitesimal Wasserstein distance
between heat kernels based at nearby points.
They proved that the first variation of $\widetilde g_t$
at $t=0$ agrees with $-2\Ric_g$ after integration along any geodesic, and
pointwise almost everywhere along it.

Let $g_t$ denote the Ricci flow with initial metric $g$.  Since
$\widetilde g_t$ is defined using the heat semigroup of the initial metric
rather than an evolving Laplacian, its first-order agreement with $g_t$
leaves open the order at which the two deformations first separate and
the geometric quantity that governs their discrepancy.  We show that the
two deformations differ at quadratic order in general.  More precisely, we
prove
\[
 \widetilde g_t-g_t=t^2D_g+O_{C^0}(t^3),
\]
where $D_g$ is the tensor defined in \eqref{eq:A-D-main}.
In the two second-order expansions computed below, the only term
involving covariant derivatives of curvature is $-\Delta\Ric_g$, and it
appears with the same coefficient in both.  The term therefore cancels in
the difference.  Hence $D_g$ is a tensor determined pointwise and
quadratically by the curvature.

The Gigli--Mantegazza construction rests on a heat flow and a metric
measure structure rather than on a smooth structure, and it has been
studied on suitable singular spaces and in weighted and warped variants
\cite{bandara2015geometric,carfora2014wasserstein,erbar2018smoothing}.
These developments motivate viewing the construction as a possible
extension of Ricci flow beyond the smooth category.  On smooth manifolds,
however, our result shows that the two deformations agree only to first
order in general.

Unless otherwise stated, all pointwise inner products,
norms, and contractions are taken with respect to $g$.
Let $\nabla$ denote the Levi--Civita connection of $g$.  We use
$\Rm_g(X,Y)Z=\nabla_X\nabla_YZ-\nabla_Y\nabla_XZ-\nabla_{[X,Y]}Z$ and
$\Ric_g(X,Y)=\sum_a\langle\Rm_g(e_a,X)Y,e_a\rangle$ for an orthonormal
basis $(e_1,\ldots,e_n)$ of $T_xM$.  We write
$\Rm_{abcd}=\langle\Rm_g(e_a,e_b)e_c,e_d\rangle$, raise indices with $g$,
and set $|\Rm_g|^2=\sum_{a,b,c,d}\Rm_{abcd}^2$.  On tensor fields, the
rough Laplacian is $\Delta=\operatorname{tr}_g\nabla^2$.  Our heat equation convention is
$\partial_tu=\Delta_gu$.

For $X,Y\in T_xM$ we define
\begin{align}
 \mathcal Q_g(X,Y)
 &=\sum_{a,b,c=1}^n
 \langle\Rm_g(e_a,e_b)e_c,X\rangle
 \langle\Rm_g(e_a,e_b)e_c,Y\rangle,
 \label{eq:Q-main}\\
 \mathcal S_g(X,Y)
 &=\sum_{a,b=1}^n\Ric_g(e_a,e_b)
 \langle\Rm_g(e_a,X)Y,e_b\rangle.
 \label{eq:S-main}
\end{align}
Let $\Ric_g^\#\colon TM\to TM$ be defined by
\[
 \langle\Ric_g^\#X,Y\rangle=\Ric_g(X,Y).
\]
We write
$\Ric_g^2(X,Y)=\langle\Ric_g^\#X,\Ric_g^\#Y\rangle$ and put
\begin{equation}\label{eq:A-D-main}
 A_g=-\Delta\Ric_g+2\Ric_g^2-\frac23\mathcal Q_g,
 \qquad
 D_g=2\mathcal S_g-\frac23\mathcal Q_g.
\end{equation}

\begin{theorem}[Second-order asymptotics]
\label{thm:main}
Let $(M,g)$ be a closed connected Riemannian manifold.  Let $g_t$ denote
the Ricci flow with initial metric $g$.  There are $T>0$ and $C<\infty$
depending on $(M,g)$ with the following property.  For every $x\in M$,
every $v\in T_xM$ with $|v|_g=1$, and every $0<t\le T$, one has
\begin{equation}\label{eq:main-GM}
 \left|\widetilde g_t(v,v)-g(v,v)
 +2t\Ric_g(v,v)-t^2A_g(v,v)\right|\le Ct^3
\end{equation}
and
\begin{equation}\label{eq:main-corrected}
 \left|\widetilde g_t(v,v)-g_t(v,v)-t^2D_g(v,v)\right|
 \le Ct^3.
\end{equation}
It follows that
\begin{equation}\label{eq:main-GH}
 \dGH\bigl((M,d_{\widetilde g_t}),(M,d_{g_t})\bigr)
 \le Ct^2.
\end{equation}
\end{theorem}

Equivalently, the remainder tensors in \eqref{eq:main-GM} and
\eqref{eq:main-corrected} are $O(t^3)$ in the $C^0$ operator norm induced
by $g$.
On the unit round sphere $S^n$ with $n\ge2$, the Gromov--Hausdorff
distance in \eqref{eq:main-GH} has leading term
$\frac{\pi(n-1)(3n-5)}{6}t^2$, so the order $t^2$ in
\eqref{eq:main-GH} is sharp.

The tensor $D_g$ also yields the following rigidity and
anisotropy consequences.
Recall that Ricci flow $g_t$ is stationary when $g$ is Ricci-flat and remains
homothetic to $g$ when $g$ is Einstein.

\begin{corollary}[Rigidity and anisotropy]
\label{cor:qualitative}
Let $(M,g)$ be a closed connected Riemannian manifold.
\begin{enumerate}
\item If $g$ is Ricci-flat, then
$\widetilde g_t=g_t-\frac23t^2\mathcal Q_g+O_{C^0}(t^3)$.
It follows that $\widetilde g_t=g_t+O_{C^0}(t^3)$ if and only if $g$ is
flat.
\item If $\dim M=2$ or $3$, then
$\widetilde g_t=g_t+O_{C^0}(t^3)$ if and only if $g$ is flat.
\item If $S^2$ and $S^3$ carry round metrics of sectional curvature
$\lambda$ and $\lambda/2$ for $\lambda>0$, respectively, then their
product metric $g$ is Einstein, but $\widetilde g_t$ is not homothetic to
$g$ for all small $t>0$.
\end{enumerate}
\end{corollary}
The statements in Corollary~\ref{cor:qualitative} are proved in
Section~\ref{sec:geometric-consequences}.

Theorem~\ref{thm:main} also refines two first-order formulas of Sturm.
For $x\in M$ and $v\in T_xM$ with $|v|_g=1$, Sturm proved that
$-\frac1{2t}\log\widetilde g_t(v,v)\to\Ric_g(v,v)$ as $t\downarrow0$
\cite[Theorem~3.5]{sturm2021remarks}.  The expansion \eqref{eq:main-GM}
refines this limit to
\begin{equation}\label{eq:Sturm-refinement}
 -\frac1{2t}\log\widetilde g_t(v,v)
 =\Ric_g(v,v)
 +t\left(\bigl(\Ric_g(v,v)\bigr)^2-\frac12A_g(v,v)\right)+O(t^2).
\end{equation}
Thus $A_g$ determines the new term in the order-$t$ correction to Sturm's
limit.  Replacing $\widetilde g_t$ by $g_t$ in the left-hand side of
\eqref{eq:Sturm-refinement} changes the coefficient of $t$ by
$D_g(v,v)/2$.

Integrating the metric expansion along a curve gives a corresponding
second-order refinement of Sturm's action formula
\cite[Corollary~3.6]{sturm2021remarks}.

\begin{corollary}[Second-order Wasserstein action]
\label{cor:Wasserstein-action}
Let $T$ and $C$ be as in Theorem~\ref{thm:main}.  For every $C^1$ curve
$\gamma\colon[0,1]\to M$ and every $0<t\le T$, let
$\mu_{t,\gamma_\cdot}$ denote the curve $s\mapsto\mu_{t,\gamma_s}$.  Then
\[
 \operatorname{Act}_2(\mu_{t,\gamma_\cdot})
 =\int_0^1(g-2t\Ric_g+t^2A_g)(\dot\gamma_s,\dot\gamma_s)\,ds+R_t(\gamma),
\]
where
$\operatorname{Act}_2(\eta)=\int_0^1|\dot\eta_s|_{W_2}^2\,ds$ for
$\eta\in AC^2\bigl([0,1];(\mathcal P_2(M),W_2)\bigr)$, and
\[
 |R_t(\gamma)|
 \le Ct^3\int_0^1|\dot\gamma_s|_g^2\,ds.
\]
\end{corollary}

For fixed $x\in M$ and a unit vector $v\in T_xM$, the quantity
$\widetilde g_t(v,v)$ is the maximum of a concave quadratic functional.
Its maximizer modulo additive constants is the velocity
potential $\varphi_{t,x,v}$ of the curve
$a\mapsto\mu_{t,\exp_x(av)}$ at $a=0$.  We approximate this maximizer by
$f_t=f_0+tf_1$, where $f_0$ is the linear function
$\langle v,\exp_x^{-1}(\cdot)\rangle$ near $x$.  After parabolic
rescaling, the choice of $f_0$ makes the zeroth-order terms agree, and
$f_1$ is chosen to make the terms of order $t$ agree.  We obtain a uniform residual
estimate for this approximation using differentiated heat-kernel
asymptotics in the rescaled region near $x$ together with a global tail
bound.
Testing the residual estimate with $f_t-\varphi_{t,x,v}$ yields
$\bigl[P_t(|\nabla f_t-\nabla\varphi_{t,x,v}|^2)(x)\bigr]^{1/2}
=O(t^{3/2})$.  The exact quadratic deficit identity expresses the metric
deficit as the square of this error, so a first-order approximation of
the velocity potential suffices to obtain the second-order metric
coefficient with an $O(t^3)$ remainder.
Expanding the variational functional at $f_t$ using the heat semigroup
determines $A_g$.  The second-order Ricci-flow expansion contains the same
term $-\Delta\Ric_g$, so subtracting the two
expansions leaves $D_g$, which depends algebraically on the curvature.

\section{Wasserstein formulation of the Gigli--Mantegazza metric}
\label{sec:background}

We recall the smooth Wasserstein formalism needed below.  Let
$\mathcal P_2(M)$ be the space of Borel probability measures on $(M,g)$.
Equip this space with the quadratic Wasserstein distance
\[
 W_2^2(\mu,\nu)
 =\inf_{\pi\in\Pi(\mu,\nu)}
 \int_{M\times M}d_g(x,y)^2\,d\pi(x,y),
\]
where $\Pi(\mu,\nu)$ denotes the set of couplings of $\mu$ and $\nu$.

Let $\mathcal P^\infty(M)\subset\mathcal P_2(M)$ denote the space of
measures $\rho\,d\operatorname{vol}_g$ with $\rho\in C^\infty(M)$
positive.  Fix $\mu=\rho\,d\operatorname{vol}_g\in\mathcal P^\infty(M)$.
A smooth tangent vector at $\mu$ has the form
$\sigma\,d\operatorname{vol}_g$ for some $\sigma\in C^\infty(M)$ satisfying
$\int_M\sigma\,d\operatorname{vol}_g=0$.  Otto's formal Riemannian
structure on $\mathcal P^\infty(M)$ represents this tangent vector by a
potential $\varphi\in C^\infty(M)$ satisfying
\[
 -\nabla\!\cdot(\rho\nabla\varphi)=\sigma.
\]
The potential is unique up to an additive constant.
Identifying tangent vectors with their potentials modulo constants, the
formal Riemannian metric at measure $\mu$ is
\[
 \langle\varphi,\psi\rangle_\mu
 =\int_M\langle\nabla\varphi,\nabla\psi\rangle_g\rho\,
 d\operatorname{vol}_g.
\]
Let $s\mapsto\rho_s\in C^\infty(M)$ be a smooth family of positive
probability densities.  Set
$\eta_s=\rho_s\,d\operatorname{vol}_g$.  Its tangent vector is represented
by the velocity potential $\varphi_s$ solving the continuity equation
\[
 -\nabla\!\cdot(\rho_s\nabla\varphi_s)=\partial_s\rho_s.
\]
The corresponding formal norm agrees with the metric speed in
$(\mathcal P_2(M),W_2)$.  Their common value satisfies
\[
 |\dot\eta_s|_{W_2}
 :=\lim_{r\to s}\frac{W_2(\eta_r,\eta_s)}{|r-s|},
 \qquad
 |\dot\eta_s|_{W_2}^2
 =\langle\varphi_s,\varphi_s\rangle_{\eta_s}.
\]
See \cite{lott2008geometric,ambrosio2013users} for background on
Wasserstein geometry.

Write
$P_tf(x)=\int_Mp_t(x,y)f(y)\,d\operatorname{vol}_g(y)$.
For $t>0$, consider the heat kernel embedding
\[
 \iota_t\colon M\longrightarrow\mathcal P^\infty(M),
 \qquad
 \iota_t(x)=\mu_{t,x}:=p_t(x,\cdot)\,d\operatorname{vol}_g.
\]
Fix $x\in M$ and $v\in T_xM$.  Let $d_1$ denote differentiation in the
first heat kernel variable while the second variable is fixed.  The tangent
vector $d\iota_t|_x(v)$ has density $d_1p_t(x,\cdot)(v)$.  It is
represented by the potential $\varphi_{t,x,v}$ solving
\begin{equation}\label{eq:velocity-equation}
 -\nabla_y\!\cdot\bigl(p_t(x,y)\nabla_y\varphi_{t,x,v}(y)\bigr)
 =d_1p_t(x,y)(v).
\end{equation}
Since $\int_Mp_t(x,y)\,d\operatorname{vol}_g(y)=1$ for every $x$,
differentiating in the first variable shows that the right-hand side of
\eqref{eq:velocity-equation} has zero integral.  Standard elliptic
theory therefore shows that the equation admits a smooth solution, unique
up to an additive constant
\cite[Theorem~2.2 and Proposition~3.1]{gigli2014flow}.
Since the formulas below depend only on its gradient, we normalize the
solution to have zero $\operatorname{vol}_g$-mean.

The Gigli--Mantegazza metric $\widetilde g_t$ is the pullback under
$\iota_t$ of the formal Riemannian metric on $\mathcal P^\infty(M)$.
For $v,w\in T_xM$, define
\[
 \widetilde g_t(v,w)
 =\langle\varphi_{t,x,v},\varphi_{t,x,w}\rangle_{\mu_{t,x}}
 =P_t\langle\nabla\varphi_{t,x,v},
             \nabla\varphi_{t,x,w}\rangle(x).
\]
Gigli and Mantegazza
\cite[Propositions~3.4 and~3.5]{gigli2014flow} prove that the tensor
$\widetilde g_t$ is a smooth Riemannian metric on $M$.  They also prove
that the curve $s\mapsto\mu_{t,\gamma_s}$ is absolutely continuous in
$(\mathcal P_2(M),W_2)$ whenever $\gamma$ is absolutely continuous in
$M$.  Its metric speed satisfies
\begin{equation}\label{eq:metric-speed-pullback}
 |\dot\mu_{t,\gamma_s}|_{W_2}^2
 =\widetilde g_t(\dot\gamma_s,\dot\gamma_s)
\end{equation}
for almost every $s$.

For $f,h\in C^\infty(M)$, set
\begin{equation}\label{eq:BL-forms}
 \mathcal B_t(f,h)
 :=\langle f,h\rangle_{\mu_{t,x}}
 =P_t\langle\nabla f,\nabla h\rangle(x),
 \qquad
 \mathcal L_t(h)
 :=\langle\varphi_{t,x,v},h\rangle_{\mu_{t,x}}
 =d(P_th)_x(v).
\end{equation}
The last equality follows by integration by parts in
\eqref{eq:velocity-equation}.  With $x$ and $v$ fixed, we suppress the
dependence of $\mathcal B_t$ on $x$ and that of $\mathcal L_t$ on $x$ and $v$.
The definition of $\mathcal L_t$ and the bilinearity of $\mathcal B_t$ give
\[
 \begin{aligned}
 2\mathcal L_t(f)-\mathcal B_t(f,f)
 &=2\mathcal B_t(\varphi_{t,x,v},f)-\mathcal B_t(f,f)\\
 &=\mathcal B_t(\varphi_{t,x,v},\varphi_{t,x,v})
   -\mathcal B_t(f-\varphi_{t,x,v},f-\varphi_{t,x,v}).
 \end{aligned}
\]
The final quadratic term equals
$P_t(|\nabla(f-\varphi_{t,x,v})|^2)(x)$, so it is nonnegative and
vanishes for $f=\varphi_{t,x,v}$.  Since
$\widetilde g_t(v,v)=\mathcal B_t(\varphi_{t,x,v},\varphi_{t,x,v})$, we obtain
\begin{equation}\label{eq:gm-definition}
 \widetilde g_t(v,v)
 =\sup_{f\in C^\infty(M)}
 \left\{2\mathcal L_t(f)-\mathcal B_t(f,f)\right\}
 =\mathcal B_t(\varphi_{t,x,v},\varphi_{t,x,v})
 =P_t(|\nabla\varphi_{t,x,v}|^2)(x).
\end{equation}
We will use this variational formula below.

\section{First-order approximation of the velocity potential}
\label{sec:first-correction}

Since $M$ is closed, choose $0<r_*<\inj(M,g)$.  Then $\exp_x$ is a
diffeomorphism on a neighborhood of
$\overline{B(0,r_*)}\subset T_xM$ for every $x\in M$.  For each $x$, choose
an orthonormal basis of $T_xM$ and let $\xi=(\xi^1,\ldots,\xi^n)$ denote
the associated geodesic normal coordinates centered at $x$.  By compactness,
we may decrease $r_*$ so that
\[
 \frac12|\zeta|^2\le g^{ij}(\xi)\zeta_i\zeta_j\le2|\zeta|^2
 \qquad (|\xi|\le r_*,\ \zeta\in\mathbb R^n)
\]
for every $x\in M$.  In these coordinates, for every $m\ge0$, all partial
derivatives of order at most $m$ of the
coordinate functions $g_{ij}$ and $g^{ij}$ and of $\Gamma^k_{ij}$ are
bounded uniformly for $x\in M$ and $|\xi|\le r_*$.

Choose $0<r_0<r_*/4$ and a smooth cutoff function
$\chi_0\colon[0,\infty)\to[0,1]$ satisfying
\[
 \chi_0(s)=
 \begin{cases}
  1, & \text{if }0\le s\le r_0,\\
  0, & \text{if }s\ge 2r_0.
 \end{cases}
\]
For $x\in M$, set $\chi_x(y)=\chi_0(d(x,y))$.  For $v\in T_xM$ and
$y\in B_g(x,2r_0)$, let $\xi=\exp_x^{-1}(y)\in T_xM$ and define
\begin{equation}\label{eq:f1-main}
 f_0(y)=\chi_x(y)\langle v,\xi\rangle,\qquad
 f_1(y)=-\frac43\chi_x(y)\langle\Ric_g^\#v,\xi\rangle.
\end{equation}
Extending both functions by zero outside this ball gives smooth functions
on $M$.  For $t>0$, set $f_t=f_0+tf_1$.  When $x$ and $v$
are fixed, we omit them from the notation for these functions.
By compactness and the uniform normal-coordinate bounds above, for every
$m\ge0$ there is $C_m<\infty$ such that
\begin{equation}\label{eq:uniform-trial-bounds}
 \sup_{\substack{x\in M,\ v\in T_xM\\ |v|_g=1}}
 \left(\|\nabla^j f_0\|_{L^\infty(M)}
 +\|\nabla^j f_1\|_{L^\infty(M)}\right)\le C_m,
 \qquad 0\le j\le m.
\end{equation}

\begin{theorem}[First-order approximation]\label{thm:first-correction}
There are $T>0$ and $C<\infty$ depending on $(M,g)$, such that for every
$x\in M$ and $v\in T_xM$ with $|v|_g=1$,
\begin{equation}\label{eq:energy-error}
 P_t\!\left(
 |\nabla\varphi_{t,x,v}-\nabla f_t|^2
 \right)(x)\le Ct^3,
 \qquad 0<t\le T.
\end{equation}
\end{theorem}

The proof is based on the following residual estimate, stated in terms of
the forms introduced in \eqref{eq:BL-forms}.

\begin{proposition}[Residual estimate]\label{prop:residual}
There are $T>0$ and $C<\infty$ depending on $(M,g)$, such that for every
$x\in M$ and $v\in T_xM$ with $|v|_g=1$,
\begin{equation}\label{eq:UR}
 |\mathcal B_t(f_t,h)-\mathcal L_t(h)|
 \le Ct^{3/2}\mathcal B_t(h,h)^{1/2}
\end{equation}
for every $h\in C^\infty(M)$ and $0<t\le T$.
\end{proposition}
Once Proposition~\ref{prop:residual} is proved, Theorem~\ref{thm:first-correction}
follows by testing \eqref{eq:UR} against the difference between the exact
and approximate velocity potentials.

\section{Heat kernel bounds and parabolic rescaling}
\label{sec:heat-kernel-input}

Since $(M,g)$ is closed, the heat kernel upper bound
\cite[Theorem~4]{cheng1981upper} and the
gradient estimate for $\log p_t$
\cite[Theorem~5.5.3]{hsu2002stochastic} ensure the existence of constants
$T,C_0,c_0,C_1>0$ such that
\begin{align}
 p_t(x,y)&\le C_0t^{-n/2}\exp\!\left(-c_0d(x,y)^2/t\right),
 \label{eq:global-gaussian}\\
 |\nabla_x\log p_t(x,y)|+|\nabla_y\log p_t(x,y)|
 &\le C_1\left(t^{-1/2}+d(x,y)/t\right)
 \label{eq:global-log-gradient}
\end{align}
for $x,y\in M$ and $0<t\le T$. For $0<t<1$, set
\[
 L_t=\left(\frac{n/2+8}{c_0}\log(1/t)\right)^{1/2}.
\]

\begin{lemma}[Heat kernel tail]\label{lem:heat-kernel-tail}
After decreasing $T$, there is $C<\infty$ depending on $(M,g)$ such
that the set
\[
 E_{t,x}=\{y\in M:d(x,y)\ge\sqrt t\,L_t\}
\]
satisfies
\begin{equation}\label{eq:heat-kernel-tail}
 \int_{E_{t,x}}\left(1+|\nabla_x\log p_t|^2
 +|\nabla_y\log p_t|^2\right)p_t(x,y)\,d\operatorname{vol}_g(y)
 \le Ct^6
\end{equation}
for every $x\in M$ and $0<t\le T$.
\end{lemma}

\begin{proof}
Decrease $T$ so that $T<1$.  Every $y\in E_{t,x}$ satisfies
$d(x,y)^2/t\ge L_t^2$, so the pointwise estimate
\[
 p_t(x,y)
 \le C_0t^{-n/2}e^{-c_0L_t^2}
 =C_0t^8,
 \qquad y\in E_{t,x},
\]
follows from the Gaussian bound~\eqref{eq:global-gaussian}.  Since $M$
has finite volume, integration over $E_{t,x}$ then gives
\begin{equation}\label{eq:tail-mass}
 \int_{E_{t,x}}p_t(x,y)\,d\operatorname{vol}_g(y)\le Ct^8.
\end{equation}
Since $T<1$, the estimate~\eqref{eq:global-log-gradient} and
$d(x,y)\le\diam(M,g)$ imply the uniform
bound
\[
 |\nabla_x\log p_t|+|\nabla_y\log p_t|
 \le Ct^{-1}.
\]
Squaring this bound shows that the integrand in
\eqref{eq:heat-kernel-tail} is at most a constant multiple of
$t^{-2}p_t(x,y)$.  Combining this with \eqref{eq:tail-mass} proves
\eqref{eq:heat-kernel-tail}.
\end{proof}

Fix $x\in M$ and $v\in T_xM$ with $|v|_g=1$.  We use the geodesic normal
coordinates $\xi=(\xi^1,\ldots,\xi^n)$ chosen in
Section~\ref{sec:first-correction}.  For $t>0$, set $z=t^{-1/2}\xi$ and
define
\[
 y_{t,x}\colon B(0,r_*/\sqrt t)\longrightarrow B_g(x,r_*),\qquad
 y_{t,x}(z)=\exp_x(\sqrt t\,z).
\]
Let $d\xi$ and $dz$ denote Lebesgue measure in the variables $\xi$ and $z$.
The Euclidean heat kernel based at the origin is
\[
 p_t^{\mathbb R^n}(0,\xi)
 =(4\pi t)^{-n/2}e^{-|\xi|^2/(4t)}.
\]
Under the substitution $\xi=\sqrt t\,z$, we have
$p_t^{\mathbb R^n}(0,\xi)\,d\xi=p_1^{\mathbb R^n}(0,z)\,dz$.
On $B(0,r_*/\sqrt t)$, define $\rho_{t,x}$ by
\[
 (y_{t,x}^{-1})_\#
 \left(\left.p_t(x,\cdot)\,d\operatorname{vol}_g\right|_{B_g(x,r_*)}\right)
 =\rho_{t,x}(z)p_1^{\mathbb R^n}(0,z)\,dz.
\]
To rescale the right-hand side of \eqref{eq:velocity-equation}, define
\[
 s_{t,x,v}(z):=\sqrt t\,
 \left[d_1\log p_t(x,y)(v)\right]_{y=y_{t,x}(z)}.
\]
Here $d_1$ differentiates in the first heat kernel variable with $y$ fixed,
and the substitution $y=y_{t,x}(z)$ is made afterward.  Since
$y_{t,x}(z)$ depends on $x$, differentiating after the substitution would
give a different quantity.  The density $\rho_{t,x}$, the inverse metric
coefficients $g^{ij}(\sqrt t\,z)$, and the right-hand side $s_{t,x,v}$ are
the quantities that enter the rescaled velocity equation derived in
\eqref{eq:rescaled-operator}.  The following lemma records their
expansions.
\begin{lemma}[Rescaled expansions]\label{lem:rescaled-expansions}
There are $T>0$ and $C<\infty$ depending on $(M,g)$, such that
the following conclusions hold for every $x\in M$ and every unit vector
$v\in T_xM$, whenever $0<t\le T$ and $|z|\le L_t$.
\begin{equation}\label{eq:rescaled-expansions}
\begin{aligned}
 \rho_{t,x}&=1+t\left(\frac16\operatorname{Scal}_g(x)
 -\frac1{12}\Ric_g(z,z)\right)+t^{3/2}r_t^\rho,
 \\
 g^{ij}(\sqrt t\,z)&=\delta^{ij}
 -\frac t3\Rm^i{}_{k}{}^j{}_{\mu}(x)z^kz^\mu+t^{3/2}r_t^{ij},
 \\
 s_{t,x,v}&=\frac12\langle v,z\rangle
 -\frac t6\Ric_g(v,z)+t^{3/2}r_t^s.
\end{aligned}
\end{equation}
The remainder functions satisfy
\begin{equation}\label{eq:parametrix-remainder-bound}
 |r_t^\rho|+|\partial r_t^\rho|
 +\max_{i,j}\bigl(|r_t^{ij}|+|\partial r_t^{ij}|\bigr)+|r_t^s|
 \le C(1+|z|)^3.
\end{equation}
Here $\partial$ denotes any first-order partial derivative with respect to
$z$.  After decreasing $T$ if necessary, the bounds
$1/2\le\rho_{t,x}\le2$ hold on the same region.
\end{lemma}

The expansion of $\rho_{t,x}$ follows from the local heat kernel parametrix
and the standard normal-coordinate expansion of the volume form.  The
expansion of $g^{ij}$ is the standard inverse-metric expansion in normal
coordinates.  The third expansion is obtained by differentiating the
logarithm of the heat kernel parametrix in the first variable while keeping
the second variable fixed.  The proof, including the uniform remainder
estimates, is given in Appendix~\ref{app:rescaled-expansions}.

\section{Proof of the first-order approximation}
\label{sec:residual}

\begin{proof}[Proof of Proposition~\ref{prop:residual}]
Fix $x\in M$ and $v\in T_xM$ with $|v|_g=1$.
By \eqref{eq:velocity-equation}, the residual obtained by replacing
$\varphi_{t,x,v}$ with $f_t$ is
\begin{equation}\label{eq:residual-def}
 \varepsilon_t(y)
 =-\frac1{p_t(x,y)}\nabla_y\!\cdot(p_t(x,y)\nabla_yf_t)
  -d_1\log p_t(x,y)(v)
 =-\frac1{p_t}\nabla_y\!\cdot\bigl(p_t\nabla_y(f_t-\varphi_{t,x,v})\bigr).
\end{equation}
We split $M$ into the ball where $d(x,y)<\sqrt t\,L_t$ and its complement.
We estimate $\varepsilon_t$ on the ball using parabolic rescaling and the
local expansions from Lemma~\ref{lem:rescaled-expansions}, and on the
complement using the heat kernel tail estimate from
Lemma~\ref{lem:heat-kernel-tail}.
Define $u\colon B(0,r_*/\sqrt t)\to\mathbb R$ by
$u(z)=t^{-1/2}f_t(y_{t,x}(z))$.  From the coordinate expression for
divergence, we obtain
\begin{equation}\label{eq:rescaled-operator}
 \mathcal A_tu(z)
 :=\sqrt t\left[-\frac1{p_t}\nabla_y\!\cdot
 (p_t\nabla_yf_t)\right](y_{t,x}(z))
 =-\frac1{\rho_{t,x}p_1^{\mathbb R^n}(0,z)}\partial_j\!\left(
 \rho_{t,x}p_1^{\mathbb R^n}(0,z)g^{ij}(\sqrt t\,z)\partial_i u\right).
\end{equation}
Since $\sqrt t\,L_t\to0$, after decreasing $T$ we have
$d(x,y_{t,x}(z))=\sqrt t\,|z|\le r_0$ whenever $0<t\le T$ and
$|z|\le L_t$.  Hence $\chi_x(y_{t,x}(z))=1$ throughout this region, and
\[
 u(z)=\langle v,z\rangle
 -\frac43t\langle\Ric_g^\#v,z\rangle.
\]
Equations~\eqref{eq:residual-def} and \eqref{eq:rescaled-operator} give
the rescaled residual
\begin{equation}\label{eq:rescaled-residual}
 \widehat\varepsilon_t(z)
 :=\sqrt t\,\varepsilon_t(y_{t,x}(z))
 =\mathcal A_tu(z)-s_{t,x,v}(z).
\end{equation}
Applying the product rule in \eqref{eq:rescaled-operator} and using
$\partial_j\log p_1^{\mathbb R^n}(0,z)=-z_j/2$, we obtain
\begin{equation}\label{eq:rescaled-operator-product}
\begin{aligned}
 \mathcal A_tu
 &=-\partial_j(g^{ij}\partial_i u)
 -g^{ij}\partial_i u\,
   \partial_j\log\!\left(\rho_{t,x}p_1^{\mathbb R^n}(0,z)\right)\\
 &=-\partial_j(g^{ij}\partial_i u)
 +\frac12z_jg^{ij}\partial_i u
 -g^{ij}\partial_i u\,\partial_j\log\rho_{t,x},
\end{aligned}
\end{equation}
where $g^{ij}=g^{ij}(\sqrt t\,z)$.  By the inverse metric expansion in
\eqref{eq:rescaled-expansions}, we have
\[
 g^{ij}(\sqrt t\,z)\partial_i u
 =v^j+t\left(-\frac43(\Ric_g^\#v)^j
 -\frac13\Rm^i{}_{k}{}^j{}_{\mu}(x)v_i z^kz^\mu\right)
 +O\bigl(t^{3/2}(1+|z|)^3\bigr).
\]
The density expansion in \eqref{eq:rescaled-expansions} also shows that
\[
 -\partial_j\log\rho_{t,x}
 =\frac t6\Ric_{j\mu}(x)z^\mu
 +O\bigl(t^{3/2}(1+|z|)^4\bigr).
\]
We substitute these two expansions into
\eqref{eq:rescaled-operator-product}.  The derivative bounds in
\eqref{eq:parametrix-remainder-bound} allow us to differentiate the
remainders.  This gives
\begin{align}
 \mathcal A_tu
 &=\frac12\langle v,z\rangle
 +t\Bigl[-\partial_j\!\left(
 -\frac43(\Ric_g^\#v)^j
 -\frac13\Rm^i{}_{k}{}^j{}_{\mu}(x)v_i z^kz^\mu\right)\notag\\
 &\quad
 +\frac12z_j\left(-\frac43(\Ric_g^\#v)^j
 -\frac13\Rm^i{}_{k}{}^j{}_{\mu}(x)v_i z^kz^\mu\right)
 +\frac16\Ric_g(v,z)\Bigr]
 +O\bigl(t^{3/2}(1+|z|)^4\bigr).
 \label{eq:weighted-divergence-expansion}
\end{align}
Here all curvature components are evaluated at the fixed point $x$.
Antisymmetry in the last two curvature indices and
$\Rm^i{}_{j}{}^j{}_{\mu}(x)=\Ric^i{}_{\mu}(x)$ imply
\begin{align*}
 z_j\Rm^i{}_{k}{}^j{}_{\mu}(x)v_i z^kz^\mu&=0,\\
 \partial_j\left(-\frac13\Rm^i{}_{k}{}^j{}_{\mu}(x)v_i z^kz^\mu\right)
 &=-\frac13\left(
 \Rm^i{}_{j}{}^j{}_{\mu}(x)v_i z^\mu
 +\Rm^i{}_{k}{}^j{}_{j}(x)v_i z^k\right)=-\frac13\Ric_g(v,z).
\end{align*}
Equation~\eqref{eq:weighted-divergence-expansion} therefore reduces to
\[
 \mathcal A_tu(z)
 =\frac12\langle v,z\rangle-\frac t6\Ric_g(v,z)
 +O\bigl(t^{3/2}(1+|z|)^4\bigr).
\]
Comparison with the third line of \eqref{eq:rescaled-expansions} using
\eqref{eq:rescaled-residual} and \eqref{eq:parametrix-remainder-bound} yields
\[
 |\widehat\varepsilon_t(z)|\le Ct^{3/2}(1+|z|)^4.
\]
This holds whenever $|z|\le L_t$.
We have
$|\varepsilon_t(y_{t,x}(z))|^2=t^{-1}|\widehat\varepsilon_t(z)|^2$.
Lemma~\ref{lem:rescaled-expansions} also gives $\rho_{t,x}\le2$.
Integration against $\rho_{t,x}(z)p_1^{\mathbb R^n}(0,z)\,dz$ therefore
yields
\begin{equation}\label{eq:central-residual-L2}
\begin{aligned}
 \int_{|z|\le L_t}
 |\varepsilon_t(y_{t,x}(z))|^2
 \rho_{t,x}(z)p_1^{\mathbb R^n}(0,z)\,dz
 &\le Ct^2\int_{|z|\le L_t}
 (1+|z|)^8p_1^{\mathbb R^n}(0,z)\,dz\\
 &\le Ct^2\int_{\mathbb R^n}
 (1+|z|)^8p_1^{\mathbb R^n}(0,z)\,dz
 \le Ct^2.
\end{aligned}
\end{equation}
The identity $f_t=f_0+tf_1$ and \eqref{eq:uniform-trial-bounds} with $m=2$
imply
$\sup_M(|\nabla f_t|+|\Delta f_t|)\le C$.
Equation~\eqref{eq:residual-def} can be written as
\[
 \varepsilon_t
 =-\Delta f_t-\langle\nabla_y\log p_t,\nabla f_t\rangle
 -d_1\log p_t(x,y)(v).
\]
Since $|v|_g=1$, this expression and the preceding uniform bound yield
\[
 |\varepsilon_t|^2
 \le C\left(1+|\nabla_x\log p_t|^2
 +|\nabla_y\log p_t|^2\right).
\]
Lemma~\ref{lem:heat-kernel-tail} implies
\[
 \int_{E_{t,x}}|\varepsilon_t|^2p_t(x,y)
 \,d\operatorname{vol}_g(y)\le Ct^6.
\]
After decreasing $T$, the map $y_{t,x}$ identifies $M\setminus E_{t,x}$
with $\{|z|<L_t\}$.  Combining this tail estimate with
\eqref{eq:central-residual-L2}, we obtain
\begin{equation}\label{eq:global-residual-L2}
 P_t(|\varepsilon_t|^2)(x)\le Ct^2.
\end{equation}
Fix $h\in C^\infty(M)$.  Multiplying \eqref{eq:residual-def} by
$h(y)p_t(x,y)$ and integrating by parts in $y$ over $M$, we obtain
\[
 \mathcal B_t(f_t,h)-\mathcal L_t(h)
 =P_t(\varepsilon_th)(x).
\]
Since $M$ is closed, the last expression in \eqref{eq:residual-def} and the
divergence theorem imply
\[
 P_t(\varepsilon_t)(x)
 =-\int_M\nabla_y\!\cdot\bigl(p_t\nabla_y
 (f_t-\varphi_{t,x,v})\bigr)\,d\operatorname{vol}_g(y)=0.
\]
Hence
\[
 P_t(\varepsilon_th)(x)
 =P_t\bigl(\varepsilon_t(h-P_th(x))\bigr)(x).
\]
Choose $K\ge0$ such that $\Ric_g\ge-Kg$.  The Bakry--\'Emery gradient
estimate \cite[Theorem~2]{vonrenesse2005transport} and Jensen's
inequality imply that every $f\in C^\infty(M)$ satisfies
\begin{equation}\label{eq:bakry-emery-gradient}
 |\nabla P_tf|^2\le e^{2Kt}P_t\bigl(|\nabla f|^2\bigr).
\end{equation}
For $0\le s\le t$ we have
\[
 \frac{d}{ds}P_s\bigl((P_{t-s}h)^2\bigr)(x)
 =2P_s\bigl(|\nabla P_{t-s}h|^2\bigr)(x)
 \le2e^{2K(t-s)}P_sP_{t-s}(|\nabla h|^2)(x)
 =2e^{2K(t-s)}\mathcal B_t(h,h).
\]
Integrating from $0$ to $t$, we obtain
\begin{equation}\label{eq:heat-poincare-bound}
 \begin{aligned}
 P_t\bigl(|h-P_th(x)|^2\bigr)(x)
 &=P_t(h^2)(x)-(P_th(x))^2\\
 &\le2\left(\int_0^t e^{2Ks}\,ds\right)\mathcal B_t(h,h)\\
 &\le2te^{2Kt}\mathcal B_t(h,h).
 \end{aligned}
\end{equation}
Cauchy--Schwarz together with \eqref{eq:global-residual-L2} and
\eqref{eq:heat-poincare-bound} yields
\begin{align*}
 |\mathcal B_t(f_t,h)-\mathcal L_t(h)|
 &=|P_t(\varepsilon_th)(x)|\\
 &=\bigl|P_t\bigl(\varepsilon_t(h-P_th(x))\bigr)(x)\bigr|\\
 &\le P_t(|\varepsilon_t|^2)(x)^{1/2}
 P_t\bigl(|h-P_th(x)|^2\bigr)(x)^{1/2}\\
 &\le Ct^{3/2}\mathcal B_t(h,h)^{1/2}.
\end{align*}
This proves \eqref{eq:UR}.
\end{proof}

\begin{proof}[Proof of Theorem~\ref{thm:first-correction}]
The velocity potential satisfies
$\mathcal B_t(\varphi_{t,x,v},h)=\mathcal L_t(h)$ for every
$h\in C^\infty(M)$.
Set $e=f_t-\varphi_{t,x,v}$.  Applying
Proposition~\ref{prop:residual} with $h=e$ yields
\[
 \mathcal B_t(e,e)=\mathcal B_t(f_t,e)-\mathcal L_t(e)
 \le Ct^{3/2}\mathcal B_t(e,e)^{1/2}.
\]
If $\mathcal B_t(e,e)=0$, there is nothing to prove.  Otherwise, dividing
by its square root and squaring yields
\[
 \mathcal B_t(f_t-\varphi_{t,x,v},f_t-\varphi_{t,x,v})
 \le Ct^3,
\]
which is \eqref{eq:energy-error}.
\end{proof}

\section{From the velocity potential to the metric}

The proof of \eqref{eq:main-GM} has two parts.  The first is the exact
quadratic identity below, which transfers the approximation of the velocity
potential to the metric.  The second part expands the variational functional
in \eqref{eq:gm-definition} at $f_t$.

Fix $x\in M$ and $v\in T_xM$ with $|v|_g=1$.  Let
$f_t=f_0+tf_1$ be the approximate velocity potential constructed in
Section~\ref{sec:first-correction} for this $x$ and $v$.
Using the forms in \eqref{eq:BL-forms}, define $\mathcal J_t$ on
$C^\infty(M)$ by
\begin{equation}\label{eq:J-functional}
 \mathcal J_t(h):=2\mathcal L_t(h)-\mathcal B_t(h,h)
 =2d(P_th)_x(v)-P_t(|\nabla h|^2)(x).
\end{equation}
This is the variational functional in \eqref{eq:gm-definition}.  Its
supremum over $C^\infty(M)$ equals $\widetilde g_t(v,v)$ and is attained at
$\varphi_{t,x,v}$.

\begin{lemma}[Quadratic deficit]\label{lem:quadratic-deficit}
Every $h\in C^\infty(M)$ satisfies
\begin{equation}\label{eq:exact-deficit}
 \widetilde g_t(v,v)-\mathcal J_t(h)
 =P_t(|\nabla\varphi_{t,x,v}-\nabla h|^2)(x).
\end{equation}
\end{lemma}

\begin{proof}
Since $\mathcal L_t(h)=\mathcal B_t(\varphi_{t,x,v},h)$, the left-hand
side equals
$\mathcal B_t(\varphi_{t,x,v},\varphi_{t,x,v})
-2\mathcal B_t(\varphi_{t,x,v},h)+\mathcal B_t(h,h)$.
This is $\mathcal B_t(\varphi_{t,x,v}-h,\varphi_{t,x,v}-h)$, which is
the right-hand side of \eqref{eq:exact-deficit}.
\end{proof}

\begin{lemma}[Derivatives of $f_0$ at $x$]\label{lem:normal-jets}
At $x$ the derivatives satisfy
\[
 \nabla f_0=v,\qquad \nabla^2f_0=0,\qquad
 \nabla\Delta f_0=-\frac23\Ric_g^\#v,\qquad
 |\nabla^3f_0|^2=\frac13\mathcal Q_g(v,v).
\]
\end{lemma}
The proof of Lemma~\ref{lem:normal-jets} is given in
Appendix~\ref{app:normal-jets}.

\begin{proposition}[Expansion of the variational functional]
\label{prop:test-functional-expansion}
There are $T>0$ and $C<\infty$ depending on $(M,g)$ such that
\[
 \left|\mathcal J_t(f_t)-g(v,v)+2t\Ric_g(v,v)-t^2A_g(v,v)\right|
 \le Ct^3
\]
for every $x\in M$, every $v\in T_xM$ with $|v|_g=1$, and every
$0<t\le T$.
Here $A_g$ is defined in \eqref{eq:A-D-main}.
\end{proposition}

\begin{proof}
We prove the estimate for $0<t\le1$.  Fix $x\in M$ and a vector
$v\in T_xM$ with $|v|_g=1$. For every $F\in C^\infty(M)$, Taylor's formula applied to
$s\mapsto P_sF$ yields
\begin{equation}\label{eq:semigroup-Taylor}
 P_tF=F+t\Delta F+\frac{t^2}{2}\Delta^2F
 +\frac12\int_0^t(t-s)^2P_s\Delta^3F\,ds.
\end{equation}
After differentiating this identity on $M$ and evaluating at $x$, we obtain
\begin{equation}\label{eq:differentiated-semigroup-Taylor}
 d(P_tF)_x(v)=dF_x(v)+t\,d(\Delta F)_x(v)+\frac{t^2}{2}d(\Delta^2F)_x(v)
 +\frac12\int_0^t(t-s)^2d(P_s\Delta^3F)_x(v)\,ds.
\end{equation}
Choose $K\ge0$ such that $\Ric_g\ge-Kg$.
For every $G\in C^\infty(M)$ and $0\le s\le1$, the gradient estimate
\eqref{eq:bakry-emery-gradient} and the $L^\infty$-contractivity of $P_s$
imply
\[
 |\nabla P_sG|\le e^{Ks}\bigl(P_s(|\nabla G|^2)\bigr)^{1/2}
 \le e^{Ks}\|\nabla G\|_\infty\le C\|\nabla G\|_\infty.
\]
This bound and \eqref{eq:uniform-trial-bounds} with $m=7$ imply that each
of the quantities
\[
 \left|\int_0^t(t-s)^2d(P_s\Delta^3f_t)_x(v)\,ds\right|,
 \qquad
 \left|\int_0^t(t-s)^2
 P_s\Delta^3(|\nabla f_t|^2)(x)\,ds\right|
\]
is bounded by $C\int_0^t(t-s)^2\,ds\le Ct^3$.
All unmarked covariant derivatives below are evaluated at $x$.
Applying \eqref{eq:differentiated-semigroup-Taylor} and
\eqref{eq:semigroup-Taylor} to the respective terms of
\eqref{eq:J-functional} with $h=f_t$ and using the preceding remainder
estimates gives
\begin{equation}\label{eq:J-expansion}
 \mathcal J_t(f_t)=2d(f_t)_x(v)-|\nabla f_t|^2
 +t\left[2d(\Delta f_t)_x(v)-\Delta|\nabla f_t|^2\right]
 +t^2\left[d(\Delta^2f_t)_x(v)-\frac12\Delta^2|\nabla f_t|^2\right]
 +O(t^3).
\end{equation}
Since $\nabla f_t=v+t\nabla f_1$, we have
\begin{equation}\label{eq:J-zero}
 2d(f_t)_x(v)-|\nabla f_t|^2
 =2\langle v+t\nabla f_1,v\rangle-|v+t\nabla f_1|^2
 =|v|^2-t^2|\nabla f_1|^2.
\end{equation}
For every $h\in C^\infty(M)$, the Bochner formula
\begin{equation}\label{eq:Bochner-used}
 \frac12\Delta|\nabla h|^2
 =|\nabla^2h|^2+\langle\nabla h,\nabla\Delta h\rangle
 +\Ric_g(\nabla h,\nabla h)
\end{equation}
implies
\begin{equation}\label{eq:J-first}
 \begin{aligned}
 2d(\Delta h)_x(v)-\Delta|\nabla h|^2
 ={}&2\langle v,\nabla\Delta h\rangle-2\langle\nabla h,\nabla\Delta h\rangle
 -2|\nabla^2h|^2-2\Ric_g(\nabla h,\nabla h)\\
 ={}&2\langle v-\nabla h,\nabla\Delta h\rangle
 -2|\nabla^2h|^2-2\Ric_g(\nabla h,\nabla h).
 \end{aligned}
\end{equation}
Substituting $h=f_t$ into \eqref{eq:J-first} and using
Lemma~\ref{lem:normal-jets}, we obtain
\begin{equation}\label{eq:J-first-ft}
\begin{aligned}
 &2d(\Delta f_t)_x(v)-\Delta|\nabla f_t|^2\\
 &=2\langle v-\nabla f_t,\nabla\Delta f_t\rangle
   -2|\nabla^2f_t|^2-2\Ric_g(\nabla f_t,\nabla f_t)\\
 &=-2t\langle\nabla f_1,\nabla\Delta f_0+t\nabla\Delta f_1\rangle
   -2t^2|\nabla^2f_1|^2
   -2\Ric_g(v+t\nabla f_1,v+t\nabla f_1)\\
 &=-2\Ric_g(v,v)-2t\langle\nabla f_1,\nabla\Delta f_0\rangle
   -4t\langle\Ric_g^\#v,\nabla f_1\rangle+O(t^2).
\end{aligned}
\end{equation}
The identity $f_t-f_0=tf_1$ and the bound
\eqref{eq:uniform-trial-bounds} imply that
\begin{equation}\label{eq:J-second-replacement}
 d(\Delta^2f_t)_x(v)-\tfrac12\Delta^2|\nabla f_t|^2
 =d(\Delta^2f_0)_x(v)-\tfrac12\Delta^2|\nabla f_0|^2+O(t).
\end{equation}
After multiplication by $t^2$ in \eqref{eq:J-expansion}, the remainder
in \eqref{eq:J-second-replacement} is $O(t^3)$.  Applying $\Delta$ to
\eqref{eq:Bochner-used} with $h=f_0$ gives
\begin{equation}\label{eq:bochner-laplacian-decomposition}
 \frac12\Delta^2|\nabla f_0|^2
 =\Delta|\nabla^2f_0|^2
 +\Delta\langle\nabla f_0,\nabla\Delta f_0\rangle
 +\Delta\bigl(\Ric_g(\nabla f_0,\nabla f_0)\bigr).
\end{equation}
By Lemma~\ref{lem:normal-jets} and the product rule, the terms on the
right-hand side satisfy
\begin{equation}\label{eq:bochner-product-rules}
\begin{aligned}
 \Delta|\nabla^2f_0|^2&=2|\nabla^3f_0|^2,\\
 \Delta\langle\nabla f_0,\nabla\Delta f_0\rangle
 &=\langle\Delta\nabla f_0,\nabla\Delta f_0\rangle
   +\langle v,\Delta\nabla\Delta f_0\rangle,\\
 \Delta\bigl(\Ric_g(\nabla f_0,\nabla f_0)\bigr)
 &=(\Delta\Ric_g)(v,v)+2\Ric_g(\Delta\nabla f_0,v).
\end{aligned}
\end{equation}
All omitted product-rule terms contain $\nabla^2f_0$ and vanish at $x$.
The contracted Ricci identity
\[
 \Delta\nabla h=\nabla\Delta h+\Ric_g^\#(\nabla h)
\]
applied first to $h=f_0$ and then to $h=\Delta f_0$ gives
\[
 \Delta\nabla f_0=\nabla\Delta f_0+\Ric_g^\#v,
 \qquad
 \Delta\nabla\Delta f_0
 =\nabla\Delta^2f_0+\Ric_g^\#(\nabla\Delta f_0).
\]
Substituting these identities into the second and third lines of
\eqref{eq:bochner-product-rules} gives
\begin{align*}
 \Delta\langle\nabla f_0,\nabla\Delta f_0\rangle
 &=|\nabla\Delta f_0|^2+\langle v,\nabla\Delta^2f_0\rangle
   +2\langle\Ric_g^\#v,\nabla\Delta f_0\rangle,\\
 \Delta\bigl(\Ric_g(\nabla f_0,\nabla f_0)\bigr)
 &=(\Delta\Ric_g)(v,v)
   +2\langle\Ric_g^\#v,\nabla\Delta f_0\rangle
   +2|\Ric_g^\#v|^2.
\end{align*}
Using these two formulas and the first line of
\eqref{eq:bochner-product-rules} in
\eqref{eq:bochner-laplacian-decomposition}, the decomposition becomes
\[
 \begin{aligned}
 \frac12\Delta^2|\nabla f_0|^2
 ={}&d(\Delta^2f_0)_x(v)+2|\nabla^3f_0|^2+|\nabla\Delta f_0|^2\\
 &+4\langle\Ric_g^\#v,\nabla\Delta f_0\rangle
   +(\Delta\Ric_g)(v,v)+2|\Ric_g^\#v|^2.
 \end{aligned}
\]
The preceding identity is equivalent to
\begin{equation}\label{eq:J-second}
 d(\Delta^2f_0)_x(v)-\frac12\Delta^2|\nabla f_0|^2
 =-2|\nabla^3f_0|^2-|\nabla\Delta f_0|^2
 -4\langle\Ric_g^\#v,\nabla\Delta f_0\rangle
 -(\Delta\Ric_g)(v,v)-2|\Ric_g^\#v|^2.
\end{equation}

Combining \eqref{eq:J-expansion} with \eqref{eq:J-zero} and
\eqref{eq:J-first-ft} and then using \eqref{eq:J-second-replacement}
yields
\begin{align*}
 \mathcal J_t(f_t)
={}&|v|^2-2t\Ric_g(v,v)+t^2\bigl[d(\Delta^2f_0)_x(v)
 -\tfrac12\Delta^2|\nabla f_0|^2\\
 &-|\nabla f_1|^2-2\langle\nabla f_1,\nabla\Delta f_0\rangle
 -4\langle\Ric_g^\#v,\nabla f_1\rangle\bigr]+O(t^3).
\end{align*}
Lemma~\ref{lem:normal-jets} and \eqref{eq:f1-main} imply
\[
 \nabla\Delta f_0=-\frac23\Ric_g^\#v,\qquad
 |\nabla^3f_0|^2=\frac13\mathcal Q_g(v,v),\qquad
 \nabla f_1=-\frac43\Ric_g^\#v.
\]
The terms in \eqref{eq:J-second} involving $\Ric_g^\#v$ are
\[
 -|\nabla\Delta f_0|^2
 -4\langle\Ric_g^\#v,\nabla\Delta f_0\rangle
 -2|\Ric_g^\#v|^2
 =\left(-\frac49+\frac83-2\right)|\Ric_g^\#v|^2
 =\frac29|\Ric_g^\#v|^2.
\]
The contribution to the coefficient of $t^2$ from
\eqref{eq:J-zero} and $t$ times \eqref{eq:J-first-ft} is
\[
 -|\nabla f_1|^2-2\langle\nabla f_1,\nabla\Delta f_0\rangle
 -4\langle\Ric_g^\#v,\nabla f_1\rangle
 =\left(-\frac{16}{9}-\frac{16}{9}+\frac{16}{3}\right)
   |\Ric_g^\#v|^2
 =\frac{16}{9}|\Ric_g^\#v|^2.
\]
Hence, we have
\begin{align*}
 \mathcal J_t(f_t)
 ={}&g(v,v)-2t\Ric_g(v,v)+t^2\left[-(\Delta\Ric_g)(v,v)
 +2|\Ric_g^\#v|^2-\frac23\mathcal Q_g(v,v)\right]
 +O(t^3).
\end{align*}
Since $\Ric_g^2(v,v)=|\Ric_g^\#v|^2$, the coefficient of $t^2$ is
$A_g(v,v)$ by \eqref{eq:A-D-main}.
The preceding estimates and \eqref{eq:uniform-trial-bounds} show that the
remainders are uniform in $x$ and in unit vectors $v\in T_xM$.  This
proves the proposition.
\end{proof}

\begin{proof}[Proof of \eqref{eq:main-GM}]
Taking $h=f_t$ in Lemma~\ref{lem:quadratic-deficit} expresses
$\widetilde g_t(v,v)-\mathcal J_t(f_t)$ as the quantity bounded in
Theorem~\ref{thm:first-correction}.  Hence
\[
 \left|\widetilde g_t(v,v)-\mathcal J_t(f_t)\right|\le Ct^3
\]
uniformly in $x$ and the unit vector $v$.
Together with Proposition~\ref{prop:test-functional-expansion}, this proves
\eqref{eq:main-GM}.
\end{proof}

\begin{proof}[Proof of Corollary~\ref{cor:Wasserstein-action}]
By \eqref{eq:metric-speed-pullback}, the metric speed of
$s\mapsto\mu_{t,\gamma_s}$ is
$\sqrt{\widetilde g_t(\dot\gamma_s,\dot\gamma_s)}$ for almost every $s$.
Hence
\[
 \operatorname{Act}_2(\mu_{t,\gamma_\cdot})
 =\int_0^1\widetilde g_t(\dot\gamma_s,\dot\gamma_s)\,ds.
\]
Applying \eqref{eq:main-GM} to
$\dot\gamma_s/|\dot\gamma_s|_g$ and using homogeneity yields
\[
 \left|
 \bigl[\widetilde g_t-(g-2t\Ric_g+t^2A_g)\bigr]
 (\dot\gamma_s,\dot\gamma_s)
 \right|
 \le Ct^3|\dot\gamma_s|_g^2.
\]
The estimate is trivial when $\dot\gamma_s=0$.  Integrating over
$s\in[0,1]$ proves the stated expansion and remainder bound.
\end{proof}

\section{Comparison with Ricci flow}

\begin{lemma}[Ricci-flow evolution of the Ricci tensor]
\label{lem:ricci-evolution}
Let $g_t$ be the Ricci flow with initial metric $g$.  Then
\[
 \left.\partial_t\Ric_{g_t}\right|_{t=0}
 =\Delta\Ric_g+2\mathcal S_g-2\Ric_g^2,
\]
where $\mathcal S_g$ is the tensor defined in \eqref{eq:S-main}.
\end{lemma}

\begin{proof}
This follows from \cite[Proposition~2.5.3]{topping2006lectures} after
converting its curvature convention to ours and using the definitions of
$\mathcal S_g$ and $\Ric_g^2$.
\end{proof}

\begin{proof}[Proof of \eqref{eq:main-corrected}]
Short-time existence and smoothness of Ricci flow on the closed manifold
$M$ ensure that there is $T>0$ and $C < \infty$ such that
\[
 \sup_{0\le s\le T}
 \left(\|\partial_s\Ric_{g_s}\|_{L^\infty(g)}
 +\|\partial_s^2\Ric_{g_s}\|_{L^\infty(g)}\right)
 \le C.
\]
Taylor's formula applied to
$\partial_tg_t=-2\Ric_{g_t}$ therefore yields
\[
 g_t
 =g-2t\Ric_g
 -t^2\left.\partial_t\Ric_{g_t}\right|_{t=0}
 +O(t^3)
\]
with a remainder bounded in $C^0(M,g)$ by $Ct^3$.  Subtracting this
expansion from \eqref{eq:main-GM} and using
Lemma~\ref{lem:ricci-evolution}, we find the coefficient of $t^2$ to be
\[
 A_g+\left.\partial_t\Ric_{g_t}\right|_{t=0}
 =\left(-\Delta\Ric_g+2\Ric_g^2-\frac23\mathcal Q_g\right)
 +\left(\Delta\Ric_g+2\mathcal S_g-2\Ric_g^2\right)
 =2\mathcal S_g-\frac23\mathcal Q_g,
\]
which is $D_g$ by \eqref{eq:A-D-main}.  This proves
\eqref{eq:main-corrected}.
\end{proof}

\begin{proof}[Proof of \eqref{eq:main-GH}]
For every $x\in M$ and $v\in T_xM$, equation~\eqref{eq:main-corrected}
and the boundedness of $D_g$ imply
\[
 \left|\widetilde g_t(v,v)-g_t(v,v)\right|
 \le Ct^2|v|_g^2.
\]
Both metrics are uniformly equivalent to $g$ for $0<t\le T$.  Comparing
the length of every smooth curve therefore shows that
\[
 \sup_{x,y\in M}
 |d_{\widetilde g_t}(x,y)-d_{g_t}(x,y)|
 \le C\diam(M,g)t^2.
\]
The identity correspondence has distortion at most this supremum,
and the fixed diameter can be absorbed into $C$.  This proves
\eqref{eq:main-GH}.
\end{proof}

\section{Geometric consequences of the algebraic discrepancy}
\label{sec:geometric-consequences}

We derive several consequences of the discrepancy tensor $D_g$.
Its trace determines the volume defect of $\widetilde g_t$ relative to
$g_t$.  For Einstein four-manifolds, the tensor $D_g$ is pointwise
proportional to $g$.  At a Ricci-flat metric the kernel of $D_g$ is the
nullity space of the curvature tensor.  In dimensions two and three, its
vanishing forces flatness.  Round spheres show that the $O(t^2)$ Gromov--Hausdorff
comparison is sharp.  We begin with the trace and the resulting volume
expansion.

\subsection{The trace and the volume defect}

Contracting \eqref{eq:S-main} and \eqref{eq:Q-main} yields
$\operatorname{tr}_g\mathcal S_g=|\Ric_g|^2$ and
$\operatorname{tr}_g\mathcal Q_g=|\Rm_g|^2$.  Hence
\eqref{eq:A-D-main} gives
\begin{equation}\label{eq:trace-D}
 \operatorname{tr}_gD_g
 =2|\Ric_g|^2-\frac23|\Rm_g|^2.
\end{equation}
As $t\downarrow0$, applying
$\det(I+h)^{1/2}=1+\frac12\operatorname{tr}h+O(|h|^2)$ to
\eqref{eq:main-corrected} gives
\[
 \frac{d\operatorname{vol}_{\widetilde g_t}}
 {d\operatorname{vol}_{g_t}}
 =1+\frac{t^2}{2}\operatorname{tr}_gD_g+O(t^3)
\]
uniformly on $M$.  Equation~\eqref{eq:trace-D} therefore gives
\begin{equation}\label{eq:relative-total-volume}
 \operatorname{Vol}(M,\widetilde g_t)
 -\operatorname{Vol}(M,g_t)
 =t^2\int_M\left(|\Ric_g|^2-\frac13|\Rm_g|^2\right)
 d\operatorname{vol}_g+O(t^3).
\end{equation}

\subsection{Einstein metrics}

Suppose that $g$ is Einstein with $\Ric_g=\lambda g$.  Ricci flow is then
the exact homothety $g_t=(1-2\lambda t)g$, so it has no term of order
$t^2$.
Thus $A_g=D_g$ is the coefficient of $t^2$ in $\widetilde g_t$ and describes
its entire second-order departure from this homothety.  This coefficient
involves a quadratic contraction of the Weyl tensor.

\begin{proposition}[Weyl decomposition of $D_g$]
\label{prop:Einstein-initial-metric}
Let $(M^n,g)$ be a closed connected Einstein manifold of dimension
$n\ge3$ with $\Ric_g=\lambda g$.  Let $\mathcal Q_g$ be the tensor defined
in \eqref{eq:Q-main}.  Let $W_g$ be the Weyl tensor.  For $x\in M$ and
$X,Y\in T_xM$, set
\[
 \mathcal Q_{W,g}(X,Y)
 =\sum_{a,b,c=1}^n
 \langle W_g(e_a,e_b)e_c,X\rangle
 \langle W_g(e_a,e_b)e_c,Y\rangle,
\]
where $(e_1,\ldots,e_n)$ is an orthonormal basis of $T_xM$.  Then
\begin{equation}\label{eq:Einstein-D}
 \mathcal Q_g=\mathcal Q_{W,g}+\frac{2\lambda^2}{n-1}g,
 \qquad
 A_g=D_g=\frac{2\lambda^2(3n-5)}{3(n-1)}g-\frac23\mathcal Q_{W,g}.
\end{equation}
Combined with \eqref{eq:main-GM}, this yields
\begin{equation}\label{eq:Einstein-GM-expansion}
 \widetilde g_t
 =(1-2\lambda t)g
 +t^2\left(
 \frac{2\lambda^2(3n-5)}{3(n-1)}g
 -\frac23\mathcal Q_{W,g}\right)+O(t^3).
\end{equation}
Moreover, the discrepancy $D_g$ vanishes if and only if
$\mathcal Q_{W,g}=\frac{(3n-5)\lambda^2}{n-1}g$.
\end{proposition}

\begin{proof}
Since $n\ge3$, the contracted second Bianchi identity implies that
$\lambda$ is constant.  Hence $\nabla\Ric_g=0$ and $\Delta\Ric_g=0$. Substituting
$\Ric_g=\lambda g$ into \eqref{eq:S-main} yields
\[
\mathcal S_g(X,Y)
=\lambda\sum_a\langle\Rm_g(e_a,X)Y,e_a\rangle
=\lambda\Ric_g(X,Y)
=\lambda^2g(X,Y).
\]
The definitions in \eqref{eq:A-D-main} therefore imply
\begin{equation}\label{eq:Einstein-A-D-intermediate}
 A_g=D_g=2\lambda^2g-\frac23\mathcal Q_g.
\end{equation}
The definition of the Weyl tensor gives
\[
 \Rm_g(e_a,e_b)e_c
 =W_g(e_a,e_b)e_c
 +\frac{\lambda}{n-1}(\delta_{bc}e_a-\delta_{ac}e_b).
\]
Since $W_g$ is trace-free, we have
\[
 \sum_{a,b,c}
 \langle W_g(e_a,e_b)e_c,X\rangle
 \langle \delta_{bc}e_a-\delta_{ac}e_b,Y\rangle=0.
\]
The same equality holds with $X$ and $Y$ interchanged.  Direct summation
also gives
\[
 \sum_{a,b,c}
 \left\langle\frac{\lambda}{n-1}
 (\delta_{bc}e_a-\delta_{ac}e_b),X\right\rangle
 \left\langle\frac{\lambda}{n-1}
 (\delta_{bc}e_a-\delta_{ac}e_b),Y\right\rangle
 =\frac{2\lambda^2}{n-1}g(X,Y).
\]
Hence
$\mathcal Q_g=\mathcal Q_{W,g}+\frac{2\lambda^2}{n-1}g$.  Substitution in
\eqref{eq:Einstein-A-D-intermediate} proves \eqref{eq:Einstein-D}.
Equation~\eqref{eq:Einstein-GM-expansion} follows from
\eqref{eq:main-GM}.  The stated characterization of $D_g=0$ follows
directly from \eqref{eq:Einstein-D}.
\end{proof}

\begin{remark}\label{rem:einstein-anisotropy}
Although Ricci flow starting from an Einstein metric is homothetic, the
second-order discrepancy $D_g$ need not be a multiple of $g$.  If
$W_g=0$, then $g$ has constant sectional curvature and
\eqref{eq:Einstein-D} gives
$D_g=\frac{2\lambda^2(3n-5)}{3(n-1)}g$.  The tensor $\mathcal Q_{W,g}$
can nevertheless make $D_g$ depend on the tangent direction, as the
following example shows.

For $\lambda>0$, equip $S^2$ and $S^3$ with round metrics of sectional
curvature $\lambda$ and $\lambda/2$, respectively.  Their product metric
$g$ on $S^2\times S^3$ is Einstein with $\Ric_g=\lambda g$.  An
$m$-dimensional space form $(N,h)$ of sectional curvature $K$ satisfies
$\mathcal Q_h=2(m-1)K^2h$. It follows that
$\mathcal Q_g$ has eigenvalues $2\lambda^2$ and $\lambda^2$ on vectors
tangent to $S^2$ and $S^3$, respectively, and $D_g$ has eigenvalues
$\frac23\lambda^2$ and $\frac43\lambda^2$ on these two subspaces and is
therefore positive definite but not a multiple of $g$.
\end{remark}

In dimension four, this anisotropy disappears because
$\mathcal Q_{W,g}$ is pointwise proportional to $g$.

\begin{corollary}[Einstein four-manifolds]
\label{cor:Einstein-four-manifold}
Let $(M^4,g)$ be a closed connected Einstein manifold with
$\Ric_g=\lambda g$.  Then
\begin{equation}\label{eq:Einstein-four-identities}
 \mathcal Q_g=\frac14|\Rm_g|^2g,
 \qquad
 A_g=D_g=\left(2\lambda^2-\frac16|\Rm_g|^2\right)g.
\end{equation}
The GM metric therefore has the expansion
\begin{equation}\label{eq:Einstein-four-GM}
 \widetilde g_t
 =(1-2\lambda t)g
 +t^2\left(2\lambda^2-\frac16|\Rm_g|^2\right)g
 +O(t^3).
\end{equation}
Thus the second-order discrepancy is pointwise conformal.
\end{corollary}

\begin{proof}
By \cite[Lemma~3.6]{catino2020bochner}, every Weyl tensor in dimension
four satisfies $\mathcal Q_{W,g}=\frac14|W_g|^2g$.  For $n=4$ the first
identity in \eqref{eq:Einstein-D} reads
$\mathcal Q_g=\mathcal Q_{W,g}+\frac23\lambda^2g$.  Taking traces and using
$\operatorname{tr}_g\mathcal Q_g=|\Rm_g|^2$ and
$\operatorname{tr}_g\mathcal Q_{W,g}=|W_g|^2$ gives
$|\Rm_g|^2=|W_g|^2+\frac83\lambda^2$.  Hence
\[
 \mathcal Q_g
 =\frac14|W_g|^2g+\frac23\lambda^2g
 =\frac14|\Rm_g|^2g,
\]
which is the first identity in \eqref{eq:Einstein-four-identities}.
Substituting it in \eqref{eq:Einstein-A-D-intermediate} proves the
second.  Equation~\eqref{eq:Einstein-four-GM} follows from
\eqref{eq:main-GM}.
\end{proof}

\subsection{Ricci-flat metrics}

Ricci flow is stationary at every Ricci-flat metric.  By contrast, the
quadratic coefficient in the expansion of $\widetilde g_t$ vanishes
identically only when $g$ is flat.

\begin{corollary}[Ricci-flat rigidity]
\label{cor:Ricci-flat-rigidity}
Let $(M,g)$ be a closed connected Ricci-flat manifold.  Then
\begin{equation}\label{eq:Ricci-flat-GM}
 \widetilde g_t
 =g-\frac23t^2\mathcal Q_g+O(t^3).
\end{equation}
The corresponding volume expansion is
\begin{equation}\label{eq:Ricci-flat-volume}
 \operatorname{Vol}(M,\widetilde g_t)
 =\operatorname{Vol}(M,g)
 -\frac{t^2}{3}\int_M|\Rm_g|^2\,d\operatorname{vol}_g+O(t^3).
\end{equation}
At every $x\in M$, the kernel of $D_g$ is the nullity space of the
curvature tensor
\begin{equation}\label{eq:Ricci-flat-nullity}
 \ker D_g(x)
 =\{v\in T_xM:\Rm_g(v,w)=0\text{ for every }w\in T_xM\}.
\end{equation}
In particular, the quadratic coefficient in \eqref{eq:Ricci-flat-GM}
vanishes identically if and only if $g$ is flat.
\end{corollary}

\begin{proof}
Since $\Ric_g=0$, we have $\mathcal S_g=0$ and hence
$D_g=-\frac23\mathcal Q_g$ by \eqref{eq:A-D-main}.  Using $g_t=g$ in
\eqref{eq:main-corrected}, we obtain \eqref{eq:Ricci-flat-GM}.
Equation~\eqref{eq:relative-total-volume} then gives
\eqref{eq:Ricci-flat-volume}.
By \eqref{eq:Q-main}, the tensor $\mathcal Q_g$ is positive semidefinite
and satisfies
\[
 \mathcal Q_g(v,v)
 =\sum_{a,b,c}\langle\Rm_g(e_a,e_b)e_c,v\rangle^2.
\]
Thus the symmetry of $\Rm_g$ implies that $\mathcal Q_g(v,v)=0$ if and only if
$\Rm_g(v,w)=0$ for every $w\in T_xM$.  This proves
\eqref{eq:Ricci-flat-nullity} and the final assertion.
\end{proof}

\begin{corollary}[Ricci-flat four-manifolds]
\label{cor:Ricci-flat-four-manifold}
Let $(M^4,g)$ be a closed connected Ricci-flat manifold.  Then
\begin{equation}\label{eq:Ricci-flat-four-GM}
 \widetilde g_t
 =\left(1-\frac16|\Rm_g|^2t^2\right)g+O(t^3).
\end{equation}
If $M$ is oriented, then
\[
 \operatorname{Vol}(M,\widetilde g_t)
 =\operatorname{Vol}(M,g)
 -\frac{32\pi^2}{3}\chi(M)t^2+O(t^3).
\]
\end{corollary}

\begin{proof}
Equation~\eqref{eq:Ricci-flat-four-GM} is the case $\lambda=0$ of
\eqref{eq:Einstein-four-GM}.  For an oriented Ricci-flat four-manifold,
the Chern--Gauss--Bonnet formula
\cite[Equation~(6.31)]{besse1987einstein} gives
\[
 \int_M|\Rm_g|^2\,d\operatorname{vol}_g=32\pi^2\chi(M).
\]
The volume formula follows from \eqref{eq:Ricci-flat-volume}.
\end{proof}

\begin{remark}
The volume defect can have either sign.  It is negative for every nonflat
Ricci-flat four-manifold by Corollary~\ref{cor:Ricci-flat-four-manifold}.
On the unit round four-sphere it is positive because $\lambda=3$ and
$|\Rm_g|^2=24$ give $D_g=14g$ by
\eqref{eq:Einstein-four-identities}.
\end{remark}

\subsection{Dimensions two and three}

\begin{corollary}[Low-dimensional rigidity]
\label{cor:low-dimensional-rigidity}
Let $(M,g)$ be a closed connected Riemannian manifold of dimension two or
three.  In dimension two let $K$ denote the Gaussian curvature.  Then
\begin{equation}\label{eq:surface-D}
 D_g=\frac23K^2g
\end{equation}
at every point.  In dimension three
\begin{equation}\label{eq:three-dimensional-D}
 D_g
 =-\frac83\Ric_g^2+\frac53\operatorname{Scal}_g\Ric_g
 +\left(\frac23|\Ric_g|^2-\frac13\operatorname{Scal}_g^2\right)g.
\end{equation}
In either dimension, the tensor $D_g$ vanishes at a point if and only if
$\Rm_g$ vanishes at that point.  Consequently, the difference
$\widetilde g_t-g_t$ is $O(t^3)$ if and only if $g$ is flat.
\end{corollary}

\begin{proof}
In dimension two $\Ric_g=Kg$ and
\[
 \mathcal S_g=K^2g,
 \qquad
 \mathcal Q_g=2K^2g.
\]
This proves \eqref{eq:surface-D}.  Since $K$ determines $\Rm_g$ in
dimension two, it also proves the pointwise rigidity statement there.

In dimension three the Weyl tensor vanishes, so
\[
\begin{aligned}
 \Rm_g(X,Y)Z
 ={}&\Ric_g(Y,Z)X-\Ric_g(X,Z)Y
 +\langle Y,Z\rangle\Ric_g^\#X-\langle X,Z\rangle\Ric_g^\#Y\\
 &-\frac{\operatorname{Scal}_g}{2}
 \bigl(\langle Y,Z\rangle X-\langle X,Z\rangle Y\bigr).
\end{aligned}
\]
Fix $x\in M$ and choose an orthonormal basis $(e_1,e_2,e_3)$ of $T_xM$
that diagonalizes $\Ric_g$.  Denote the eigenvalues by
$\lambda_1,\lambda_2,\lambda_3$.
Substituting the preceding decomposition into \eqref{eq:S-main} and
\eqref{eq:Q-main} and contracting in this basis yields
\[
\begin{aligned}
 \mathcal S_g
 &=-2\Ric_g^2+\frac32\operatorname{Scal}_g\Ric_g
 +\left(|\Ric_g|^2-\frac12\operatorname{Scal}_g^2\right)g,\\
 \mathcal Q_g
 &=-2\Ric_g^2+2\operatorname{Scal}_g\Ric_g
 +\left(2|\Ric_g|^2-\operatorname{Scal}_g^2\right)g.
\end{aligned}
\]
Together with \eqref{eq:A-D-main}, these identities prove
\eqref{eq:three-dimensional-D}.  Since \eqref{eq:three-dimensional-D}
expresses $D_g$ as a polynomial in $\Ric_g$ and $g$, the tensor $D_g$ is
diagonal in the chosen basis.  Writing $D_i=D_g(e_i,e_i)$, we obtain from
\eqref{eq:three-dimensional-D} that
\begin{equation}\label{eq:three-dimensional-D-differences}
 D_i-D_j
 =-\frac13(\lambda_i-\lambda_j)
 \bigl(3\lambda_i+3\lambda_j-5\lambda_k\bigr)
\end{equation}
whenever $\{i,j,k\}=\{1,2,3\}$.
Suppose that $D_g=0$.  If the three eigenvalues were distinct, applying
\eqref{eq:three-dimensional-D-differences} to the pairs $(1,2)$ and
$(1,3)$ would give $\lambda_2=\lambda_3$, a contradiction.  Hence at least two
eigenvalues agree.  If $\lambda_1=\lambda_2=a$ and $\lambda_3=b$ with $a\ne b$, then \eqref{eq:three-dimensional-D-differences} gives
$b=2a/3$, while $D_1=0$ gives $28a^2/27=0$, again a contradiction.
Thus all three eigenvalues are equal to some $a$.  Then
$D_g=\frac43a^2g$.  Since $D_g=0$, we have $a=0$.
Thus $D_g=0$ implies $\Ric_g=0$ and
$\Rm_g=0$ in dimension three.  If $\Rm_g=0$, the definitions of $\mathcal S_g$ and
$\mathcal Q_g$ give $D_g=0$.  Equation~\eqref{eq:main-corrected} now shows
that the difference $\widetilde g_t-g_t$ is $O(t^3)$ in the $C^0$ operator
norm induced by $g$ if and only if $D_g$ vanishes identically.
\end{proof}

\subsection{Sharpness on round spheres}

\begin{corollary}[Expansion on the unit round sphere]\label{cor:sphere}
For $n\ge2$, let $g_{S^n}$ be the unit round metric on $S^n$.  Then
\begin{equation}\label{eq:sphere-metric}
 \widetilde g_t
 =\left[1-2(n-1)t
 +\frac{2(n-1)(3n-5)}3t^2+O_n(t^3)\right]g_{S^n}.
\end{equation}
Here the constant implicit in $O_n(t^3)$ depends only on $n$.
The corresponding Gromov--Hausdorff expansion is
\begin{equation}\label{eq:sphere-GH}
 \dGH\bigl((S^n,d_{\widetilde g_t}),
 (S^n,d_{g_t})\bigr)
 =\frac{\pi(n-1)(3n-5)}6t^2+O_n(t^3).
\end{equation}
Thus the $O(t^2)$ estimate in \eqref{eq:main-GH} is sharp.
\end{corollary}

\begin{proof}
For $n\ge3$, the unit round sphere is Einstein with $\lambda=n-1$ and has
vanishing Weyl tensor.  Thus \eqref{eq:Einstein-GM-expansion} implies
\eqref{eq:sphere-metric}.  In dimension two the unit round metric has $\Ric_g=g$, and the space-form
identity in Remark~\ref{rem:einstein-anisotropy} gives $\mathcal Q_g=2g$.
The identity $\Ric_g=g$ gives $\Delta\Ric_g=0$ and $\Ric_g^2=g$.  The definition~\eqref{eq:A-D-main}
therefore yields $A_g=\frac23g$, and \eqref{eq:main-GM} implies
\eqref{eq:sphere-metric}.  The isometry group preserves the heat kernel
and acts transitively on the unit tangent bundle.  Hence $\widetilde g_t$
is a scalar multiple of $g_{S^n}$ for every $t$, so both
$d_{\widetilde g_t}$ and $d_{g_t}$ are multiples of $d_{g_{S^n}}$.

Since Ricci flow is $g_t=(1-2(n-1)t)g_{S^n}$, taking square roots of the
two scalar factors yields
\[
 d_{\widetilde g_t}-d_{g_t}
 =\left[\frac{(n-1)(3n-5)}3t^2+O_n(t^3)\right]d_{g_{S^n}}.
\]
Since the round sphere has diameter $\pi$, this proves
\eqref{eq:sphere-GH}.
\end{proof}

\appendix
\section{Proof of the rescaled expansions}
\label{app:rescaled-expansions}

\begin{proof}[Proof of Lemma~\ref{lem:rescaled-expansions}]
In the normal coordinates used in Section~\ref{sec:first-correction},
define the volume density $\mathfrak j_x$ by
\[
 (\exp_x)^*d\operatorname{vol}_g=\mathfrak j_x(\xi)\,d\xi.
\]
When $d(x,y)<\inj(M,g)$, both $\exp_x^{-1}(y)$ and $d(x,y)^2$ are smooth
in $(x,y)$.  For $m\in\mathbb N_0$, the local heat kernel parametrix of
order $m$ is
\[
 p_t^{\mathbb R^n}(0,\xi)\sum_{k=0}^m t^ku_k(x,y),
 \qquad \xi=\exp_x^{-1}(y).
\]
The smooth coefficients $u_k$ are determined recursively
\cite[Section~11]{li2012geometric}.  Define
\[
 \alpha_t(x,y)=\frac{p_t(x,y)}{p_t^{\mathbb R^n}(0,\xi)}.
\]
Since $\sqrt tL_t\to0$, the strong heat kernel asymptotics
\cite[Theorem~1.1]{ludewig2019strong} give constants $T,C>0$ with the
  following property.  Suppose that $0<t\le T$ and $y=y_{t,x}(z)$ with
$|z|\le L_t$.  Then
\begin{equation}\label{eq:rescaled-two-term}
 |\alpha_t-u_0-tu_1|+|\nabla_x(\alpha_t-u_0-tu_1)|
 +|\nabla_y(\alpha_t-u_0-tu_1)|\le Ct^2.
\end{equation}
The derivatives in \eqref{eq:rescaled-two-term} are taken before the
substitution $y=y_{t,x}(z)$.

The first transport equation of the parametrix construction identifies
the leading coefficient with the inverse square root of the volume
density.  The next coefficient has the classical diagonal value shown
below.
\begin{equation}\label{eq:u0}
 u_0(x,\exp_x\xi)=\mathfrak j_x(\xi)^{-1/2},
 \qquad
 u_1(x,x)=\frac16\operatorname{Scal}_g(x).
\end{equation}
Both formulas are in
\cite[Chapter~VI, Section~3]{chavel1984eigenvalues}.  There the density is
written in geodesic polar coordinates as $\mathfrak j_x=\sqrt g/r^{n-1}$,
which is the ratio of the Riemannian volume element to the Euclidean one.
The uniform normal-coordinate bounds from
Section~\ref{sec:first-correction}, the smoothness of $u_1$, and
compactness imply the expansions
\begin{align*}
 \mathfrak j_x(\xi)^{1/2}
 &=1-\frac1{12}\Ric_g(\xi,\xi)+O(|\xi|^3),\\
 u_1(x,\exp_x\xi)&=\frac16\operatorname{Scal}_g(x)+O(|\xi|),\\
 g^{ij}(\xi)&=\delta^{ij}
 -\frac13\Rm^i{}_{k}{}^j{}_{\mu}(x)\xi^k\xi^\mu
 +O(|\xi|^3).
\end{align*}
The constants in these Taylor remainders and in their first
$\xi$-derivatives are uniform in $x\in M$.
Since $d\xi=t^{n/2}dz$ and
$p_1^{\mathbb R^n}(0,z)=t^{n/2}p_t^{\mathbb R^n}(0,\sqrt t\,z)$,
the definition of $\rho_{t,x}$ implies
\[
 \rho_{t,x}(z)
 =\frac{t^{n/2}p_t(x,y_{t,x}(z))\mathfrak j_x(\sqrt t\,z)}
 {p_1^{\mathbb R^n}(0,z)}
 =\mathfrak j_x(\sqrt t\,z)\alpha_t(x,y_{t,x}(z)).
\]
By \eqref{eq:rescaled-two-term},
we have $\alpha_t=u_0+tu_1+O(t^2)$ on the rescaled region.  Combining this with
\eqref{eq:u0}, we obtain
\begin{align*}
 \rho_{t,x}(z)
 &=\mathfrak j_x(\sqrt t\,z)^{1/2}
   +t\mathfrak j_x(\sqrt t\,z)u_1(x,\exp_x(\sqrt t\,z))
   +O(t^2)\\
 &=1+t\left(\frac16\operatorname{Scal}_g(x)
             -\frac1{12}\Ric_g(z,z)\right)
   +t^{3/2}r_t^\rho(z).
\end{align*}
The last equality defines $r_t^\rho$.  The preceding Taylor estimates
and \eqref{eq:rescaled-two-term} imply
\[
 |r_t^\rho|+|\partial r_t^\rho|\le C(1+|z|)^3.
\]
This proves the first line of \eqref{eq:rescaled-expansions}.  The
normal-coordinate expansion above gives
\[
 g^{ij}(\sqrt t\,z)
 =\delta^{ij}-\frac t3\Rm^i{}_{k}{}^j{}_{\mu}(x)z^kz^\mu
  +t^{3/2}r_t^{ij}(z),
\]
where
\[
 |r_t^{ij}|\le C|z|^3,
 \qquad |\partial r_t^{ij}|\le C|z|^2.
\]
This proves the second line of \eqref{eq:rescaled-expansions}.

It remains to expand $s_{t,x,v}$.  The first heat kernel variable is
differentiated before the substitution $y=y_{t,x}(z)$.  Let $x(s)$ satisfy
$x(0)=x$ and $\dot x(0)=v$.  Keeping $y$ fixed, set
$\xi(s)=\exp_{x(s)}^{-1}(y)$.  Since $|\xi|^2=d(x,y)^2$, the definition of
$\alpha_t$ gives
\[
 \log p_t(x,y)
 =-\frac n2\log(4\pi t)-\frac{d(x,y)^2}{4t}+\log\alpha_t(x,y).
\]
The first variation of squared distance yields
\[
 d_1\!\left(-\frac{d(x,y)^2}{4t}\right)(v)
 =\frac{\langle\xi,v\rangle}{2t}.
\]
For the amplitude term, the identity~\eqref{eq:u0} and the
normal-coordinate expansion of $\mathfrak j_x$ give
\[
 \log u_0(x,\exp_x\xi)
 =\frac1{12}\Ric_g(\xi,\xi)+O(|\xi|^3).
\]
Let $D_s\xi(s)$ denote the covariant derivative along $x(s)$.  Since $y$
is fixed, we have $D_s\xi(0)=-v+O(|\xi|)$.  Indeed, the differential of
$x\mapsto\exp_x^{-1}(y)$ at $y=x$ is $-\operatorname{Id}$.  Smoothness
then implies the stated estimate.
Differentiating the preceding expansion therefore yields
\begin{align*}
 d_1\log u_0(x,y)(v)
 &=\frac1{12}(\nabla_v\Ric_g)(\xi,\xi)
   +\frac16\Ric_g\bigl(\xi,D_s\xi(0)\bigr)+O(|\xi|^2)\\
 &=-\frac16\Ric_g(v,\xi)+O(|\xi|^2).
\end{align*}
The $C^1$ estimate~\eqref{eq:rescaled-two-term} and the positivity of
$u_0$ imply
\[
 d_1\log\alpha_t(x,y)(v)
 =d_1\log u_0(x,y)(v)+O(t)
\]
on the rescaled region.  Combining the last three estimates, we obtain
\[
 d_1\log p_t(x,y)(v)
 =\frac{\langle\xi,v\rangle}{2t}
  -\frac16\Ric_g(v,\xi)+O(|\xi|^2)+O(t).
\]
We now substitute $\xi=\sqrt t\,z$ and multiply by $\sqrt t$ to obtain
\[
 s_{t,x,v}(z)
 =\frac12\langle v,z\rangle
  -\frac t6\Ric_g(v,z)+t^{3/2}r_t^s(z),
 \qquad |r_t^s|\le C(1+|z|)^2.
\]
This proves the third line of \eqref{eq:rescaled-expansions}.  Together
with the bounds above it also proves
\eqref{eq:parametrix-remainder-bound}.  Finally, note that
$t(1+L_t)^3\to0$, so we may decrease $T$ so that
$1/2\le\rho_{t,x}\le2$.
\end{proof}

\section{Derivatives of the trial function at the base point}
\label{app:normal-jets}

\begin{proof}[Proof of Lemma~\ref{lem:normal-jets}]
The cutoff equals one near $x$.  Hence for every $w\in T_xM$, the radial
geodesic $\gamma_w(s)=\exp_x(sw)$ satisfies
$f_0(\gamma_w(s))=s\langle v,w\rangle$ near $s=0$.
Since the velocity field of a geodesic is parallel, differentiating this
identity along $\gamma_w$ at $s=0$ yields
\[
 \langle\nabla f_0(x),w\rangle=\langle v,w\rangle,\qquad
 (\nabla^2f_0)_x(w,w)=0,\qquad
 (\nabla^3f_0)_x(w,w,w)=0.
\]
These identities hold for every $w\in T_xM$.  The first is equivalent to
$\nabla f_0(x)=v$, and the second implies $\nabla^2f_0(x)=0$ because the
Hessian is symmetric.  All covariant derivatives of $f_0$ from this point
onward are evaluated at $x$.

Fix any orthonormal basis $(e_1,\ldots,e_n)$ of $T_xM$ and write
$v=\sum_\mu v^\mu e_\mu$.  Set
\[
 \Theta_{aij}:=(\nabla_{e_a}\nabla^2f_0)(e_i,e_j)
 =(\nabla^3f_0)(e_a,e_i,e_j).
\]
The symmetry of the Hessian gives $\Theta_{aij}=\Theta_{aji}$.
The Ricci identity applied to the one-form $df_0$ reads
\[
 \Theta_{aij}-\Theta_{iaj}
 =-\sum_\mu\Rm_{aij\mu}v^\mu.
\]
Substituting $w=re_a+se_i+te_j$ into $\nabla^3f_0(w,w,w)=0$ gives a
polynomial in $r,s,t$ that vanishes identically.  Its coefficient of $rst$ is therefore zero.  Hence
\[
 \Theta_{aij}+\Theta_{aji}+\Theta_{iaj}
 +\Theta_{ija}+\Theta_{jai}+\Theta_{jia}=0.
\]
Pairing terms by the Hessian symmetry reduces this identity to
\[
 \Theta_{aij}+\Theta_{ija}+\Theta_{jai}=0.
\]
We therefore have
\[
 3\Theta_{aij}
 =(\Theta_{aij}-\Theta_{ija})+(\Theta_{aij}-\Theta_{jai}).
\]
Rewriting the two differences by the Hessian symmetry and applying the
Ricci identity give
\begin{equation}\label{eq:third-jet}
 \Theta_{aij}
 =-\frac13\sum_\mu(\Rm_{aij\mu}+\Rm_{aji\mu})v^\mu.
\end{equation}
Tracing over the last two indices yields
\[
 \sum_i\Theta_{aii}=-\frac23\sum_\mu\Ric_{a\mu}v^\mu.
\]
Since covariant differentiation commutes with contraction, we have
$\langle\nabla\Delta f_0,e_a\rangle=\sum_i\Theta_{aii}$ at $x$.  This proves
$\nabla\Delta f_0=-\frac23\Ric_g^\#v$.

For the norm, we substitute \eqref{eq:third-jet} into
$|\nabla^3f_0|^2=\sum_{a,i,j}\Theta_{aij}^2$ and obtain
\begin{equation}\label{eq:third-jet-norm}
 |\nabla^3f_0|^2
 =\frac19\sum_{\mu,\nu}v^\mu v^\nu
 \sum_{a,i,j}(\Rm_{aij\mu}+\Rm_{aji\mu})
 (\Rm_{aij\nu}+\Rm_{aji\nu}).
\end{equation}
Set
\[
 A_{\mu\nu}=\sum_{a,i,j}\Rm_{aij\mu}\Rm_{aij\nu},\qquad
 B_{\mu\nu}=\sum_{a,i,j}\Rm_{aij\mu}\Rm_{aji\nu}.
\]
Expanding the product in the inner sum gives two diagonal terms and two
cross terms.  Relabeling $i$ and $j$ in the second diagonal term and in one
of the cross terms yields
\[
 \sum_{a,i,j}(\Rm_{aij\mu}+\Rm_{aji\mu})
 (\Rm_{aij\nu}+\Rm_{aji\nu})
 =2A_{\mu\nu}+2B_{\mu\nu}.
\]
The first Bianchi identity also gives
\[
\begin{aligned}
 0={}&\sum_{a,i,j}
 (\Rm_{aij\mu}+\Rm_{ija\mu}+\Rm_{jai\mu})
 (\Rm_{aij\nu}+\Rm_{ija\nu}+\Rm_{jai\nu})\\
 ={}&3A_{\mu\nu}-6B_{\mu\nu}.
\end{aligned}
\]
Here the second equality follows by cyclically relabeling $a,i,j$ and using
antisymmetry in the first two curvature indices.  Thus
$2B_{\mu\nu}=A_{\mu\nu}$, and the preceding inner sum equals
$3A_{\mu\nu}$.  Finally,
\[
 \sum_{\mu,\nu}A_{\mu\nu}v^\mu v^\nu
 =\sum_{a,i,j}
 \left(\sum_\mu\Rm_{aij\mu}v^\mu\right)^2
 =\mathcal Q_g(v,v).
\]
Using these identities in \eqref{eq:third-jet-norm} proves
$|\nabla^3f_0|^2=\frac13\mathcal Q_g(v,v)$.
\end{proof}

\end{document}